\documentclass[12pt, reqno]{amsart}

\usepackage{amssymb}
\usepackage{tikz-cd}
\usepackage{xfrac}
\usepackage{thmtools}
\usepackage{thm-restate}

\author {C. Rosendal}
\address{Department of Mathematics\\University of Maryland\\4176 Campus Drive - William E. Kirwan Hall\\College Park, MD 20742-4015\\USA}
\email{rosendal@umd.edu}
\urladdr{sites.google.com/view/christian-rosendal/}

\author {J. Zomback}
\address{Department of Mathematics\\University of Maryland\\4176 Campus Drive - William E. Kirwan Hall\\College Park, MD 20742-4015\\USA}
\email{zomback@umd.edu}
\urladdr{https://sites.google.com/umd.edu/zomback/}

\usepackage{amssymb,thmtools}

\usepackage[alphabetic, nobysame]{amsrefs}

\usepackage{hyperref}
\hypersetup{
	pdfstartview={XYZ null null 1.00}, 
	pdfpagemode=UseNone, 
	colorlinks,
	breaklinks, 
	linkcolor=blue,
	urlcolor=blue, 
	anchorcolor=blue,
	citecolor=blue
}

\usepackage[capitalise]{cleveref}

\usepackage{geometry}
\usepackage{etoolbox}

\usepackage{thmtools}
\usepackage{thm-restate}
\declaretheorem[name=Theorem,numberwithin=section]{thm}

\makeatletter
\def\subsection{\@startsection{subsection}{2}%
  \z@{.8\linespacing\@plus.7\linespacing}{.5\linespacing}%
  {\normalfont\bfseries}}

\def\@seccntformat#1{%
  \protect\textup{%
    \protect\@secnumfont
    \expandafter\protect\csname format#1\endcsname 
    \csname the#1\endcsname
    \protect\@secnumpunct
  }%
}

\patchcmd{\@startsection}
  {\@afterindenttrue}
  {\@afterindentfalse}
  {}{}
 \makeatother

\makeatletter
\def\l@section{\@tocline{1}{5pt}{0pc}{}{}}
\renewcommand{\tocsection}[3]{%
	\indentlabel{\@ifnotempty{#2}{\makebox[20pt][l]{%
				\ignorespaces#1 #2.\hfill}}}\sc #3\dotfill}

\newdimen{\tocsubsecmarg}
\def\l@subsection{\@tocline{2}{3pt}{0pc}{\tocsubsecmarg}{}}
\renewcommand{\tocsubsection}[3]{%
	\indentlabel{\@ifnotempty{#2}{\makebox[30pt][l]{%
				\ignorespaces#1 #2.\hfill}}}#3\dotfill}
\makeatother
\let\oldtocsubsection=\tocsubsection
\renewcommand{\tocsubsection}[2]{\hspace{3em} \oldtocsubsection{#1}{#2}}

\numberwithin{equation}{section}

\theoremstyle{plain}

\newtheorem{lemma}[equation]{Lemma}

\crefname{prop}{Proposition}{Propositions}
\newtheorem{prop}[equation]{Proposition}

\newtheorem{theorem}[equation]{Theorem}

\crefname{obs}{Observation}{Observations}

\crefname{cor}{Corollary}{Corollaries}
\newtheorem{cor}[equation]{Corollary}

\makeatletter
\def\@empty{}
\def\ifemptycredit#1{%
	\def\tmp{#1}%
	\ifx\tmp\@empty%
	\else%
	{~(#1)}%
	\fi%
}
\makeatother

\newenvironment{namedthm*}[2][]{
\medskip\par\noindent \textbf{#2}\ifemptycredit{#1}\textbf{.}\itshape\xspace
}{\medskip}

\theoremstyle{definition}

\crefname{defn}{Definition}{Definitions}
\newtheorem{defn}[equation]{Definition}

\crefname{problem}{Problem}{Problem}
\newtheorem{problem}[equation]{Problem}

\crefname{example}{Example}{Examples}
\newtheorem{example}[equation]{Example}

\crefname{question}{Question}{Question}
\newtheorem{question}[equation]{Question}

\theoremstyle{remark}

\crefname{remark}{Remark}{Remarks}
\newtheorem{remark}[equation]{Remark}

\newtheorem{claim}[equation]{Claim}

\declaretheoremstyle[
spaceabove=\topsep, 
spacebelow=6pt,
headfont=\normalfont\itshape,
notefont=\normalfont, notebraces={(}{)},
bodyfont=\normalfont,
postheadspace=4pt,
qed=\mbox{\smaller[4]$\boxtimes$}
]{claimproofstyle}

\newcommand{\N}{\mathbb{N}}
\newcommand{\Q}{\mathbb{Q}}
\newcommand{\R}{\mathbb{R}}
\newcommand{\U}{\mathbb{U}}
\newcommand{\Z}{\mathbb{Z}}

\newcommand{\norm}[1]{\lVert#1\rVert}
\newcommand{\Norm}[1]{\big\lVert#1\big\rVert}

\newcommand{\NORMM}[1]{\bigg\lVert#1\bigg\rVert}

\newcommand{\equi}{\Leftrightarrow}
\newcommand{\inv}{^{-1}}

\renewcommand {\a} {\forall}

\newcommand{\maps}[1]{\mathop{\overset{#1}\longrightarrow}}

\newcommand{\mathes}[1]{\[\begin{split}{#1}\end{split}\]}

\newcommand{\conv}[2]{\mathop{\underset{#2}{\overset{#1}\longrightarrow}}}

\newcommand*{\defeq}{\mathrel{\vcenter{\baselineskip0.5ex \lineskiplimit0pt \hbox{\scriptsize.}\hbox{\scriptsize.}}}=}

\newcommand{\set}[1]{\left\{ #1 \right\}}

\title{Asymptotically spherical groups}

\date{\today}

\thanks{C.R. was partially supported by the U.S. National Science Foundation under Grant Number DMS-2246986. The authors are grateful to Sebastian Hurtado and Denis Osin for inspiration and insightful comments.}

\begin{document}

\begin{abstract}
    We define a notion of \emph{asymptotically spherical} topological groups, which says that spheres of large radius with respect to any maximal length function are still spherical with respect to any other maximal length function.
    This is a strengthening of a related condition introduced by Sebastian Hurtado, which we call bounded eccentricity.
    Our main result is a partial characterization of which groups are asymptotically spherical, and we also give an example of a discrete, bounded eccentric group who fails to be asymptotically spherical.
\end{abstract}

\maketitle

\tableofcontents

\section{Maximal length functions}
A {\em length function} on a group $G$ is a function $\ell\colon G\to [0,\infty[$ satisfying, for all $x,y\in G$,
\begin{enumerate}
    \item\label{pseudolength} $\ell(x)=0$ iff $x=1$,
    \item $\ell(x\inv)=\ell(x)$,
    \item $\ell(xy)\leqslant \ell(x)+\ell(y).$
\end{enumerate}
We note that there is a bijective correspondence between length functions $\ell$  and left-invariant metrics $d$ on $G$ given by the formulas
\mathes{
d(x,y)=\ell(x\inv y) \qquad\text{and}\qquad
\ell(x)=d(x,1).
}

If, furthermore, $G$ is a topological group, we say that a length function $\ell$ on $G$ is {\em compatible} provided that the corresponding metric $d$ induces the topology of $G$ or, equivalently, if the open balls 
$$
B_\ell(r)=\{x\in G\;|\; \ell(x)<r\}
$$
form a neighbourhood basis at the identity $1\in G$. Observe that, in this case, not only is $G$ metrisable as a topological space, but the topology is given by a left-invariant metric. In fact, by a result due independently to G. Birkhoff \cite{Birkhoff} and S. Kakutani \cite{Kakutani}, a topological group is metrisable if and only if it admits a left-invariant metric and these conditions are also equivalent to $G$ being Hausdorff and first countable. We shall henceforth restrict our attention to metrisable topological groups, which by our discussion are those that admit compatible length functions. Note that this class includes all discrete groups, on which all length functions taking integral values are compatible.

A  length function $\ell$ on $G$ is said to be {\em maximal} provided that it is compatible and that, for all other compatible length functions $\ell'$ on $G$, we have that 
$$
\ell'\leqslant K\ell+C
$$
for some constants $K$ and $C$. Thus, if $\ell$ and $\ell'$ are both maximal length functions on $G$, then $\ell$ and $\ell'$ are {\em quasi-isometric}, meaning that 
$$
\tfrac 1K\ell-C\leqslant \ell'\leqslant K\ell+C
$$
for some appropriate constants $K$ and $C$. Whereas not every metrisable topological group admits a maximal length function, it is possible both to give intrinsic characterisations of maximal length functions \cite[Proposition 2.72]{Rosendal} and characterise the {\em Polish}, i.e., separable and completely metrisable, topological groups that admit such \cite[Theorem 2.73]{Rosendal}. In this context, let us point out that maximal length functions and the corresponding left-invariant metrics induce a canonical quasimetric structure on the group, i.e., a canonical large scale geometric structure. In particular, when $G$ admits a maximal length function, we may unambiguously write 
$$
x\to \infty
$$
to indicate that $\ell(x)\to \infty$ for any maximal length function $\ell$. Of course, for example, when $G$ is compact, then every compatible length function is bounded. More generally, a metrisable topological group is said to be {\em globally bounded} in case every compatible length function is bounded.  Therefore, as we are interested in the asymptotics of length functions as $x\to \infty$, we will exclusively consider  groups that admit unbounded maximal  length functions and thus on which all maximal length functions are unbounded.

Our aim with this paper is to exhibit and study a rigidity phenomenon for a range of topological groups, which seems to have been overlooked before. The following definition is central to this aim.

\begin{defn}
Let $G$ be a metrisable topological group admitting unbounded maximal length functions. We say that $G$ is {\em asymptotically spherical} if, for any two maximal length functions $\ell_1$ and $\ell_2$, the limit
$$
\lambda=\lim_{x\to \infty}\frac{\ell_1(x)}{\ell_2(x)}
$$
exists. This means that, for all $\eta>1$ and all sufficiently large radii $r$, we have 
$$
B_{\ell_2}\big(\tfrac r\eta\big)\;\;\subseteq\;\; B_{\ell_1}(\lambda r)\;\;\subseteq\;\; B_{\ell_2}\big(\eta r\big)
$$
and thus asymptotically the $\ell_2$-spheres of $G$ are  $\ell_1$-spherical, which explains our terminology.
\end{defn}

The main result of our paper is a result that allows us to identify a number of asymptotically spherical groups via their actions on metric spaces. To explain this, we need to expand our terminology a bit. Namely, a map $\phi\colon M\to X$ between two metric spaces is said to be a {\em quasi-isometric embedding} if there are constants $C,K$ so that, for all $a,b\in M$, 
$$
\tfrac 1Kd(a,b)-C\leqslant d\big(\phi(a),\phi(b)\big)\leqslant Kd(a,b)+C.
$$
If $G$ is a metrisable topological group admitting maximal length functions, then a map $\phi\colon G\to X$ is a {\em quasi-isometric embedding} if it is a quasi-isometric embedding with respect to the left-invariant metric associated with a maximal length function. This is clearly independent of the specific choice of maximal length function.

\begin{restatable}{thm}{main}\label{thm:main}
    Suppose $G$ is a metrisable topological group and $G\curvearrowright X$  an isometric action on a metric space so that, for some $a\in X$, the orbital map $g\mapsto ga$ is a quasi-isometric embedding.
Assume also that 
\mathes{
\a \eta>1\;\;\exists R\;\; \a x,y\in X\;\;\inf\bigg(\sum_{i=1}^nd(z_{i-1},z_{i})\;\bigg|\; z_0=x,\; z_n=y\;\&\; d(z_{i-1},z_i)<R\bigg)<\eta\cdot d(x,y)
}
and that, for some increasing function $\sigma\colon \R_+\to \R_+$ and all $x,y,z,u\in  X$ with $d(x,y)\leqslant d(z,u)$, we have 
$$
\inf_{h\in G}\;d_{\sf Hausdorff}\big(\{h(z),h(u)\},\{x,y\}\big)< \sigma\big(d(z,u)-d(x,y)\big).
$$
Then $G$ is asymptotically spherical.
\end{restatable}

\begin{restatable}{cor}{cor:main}\label{cor:main}
The following topological groups are asymptotically spherical.
\begin{enumerate}
\item $\Z$, $D_\infty$, $\R$, 
\item ${\sf O}(n)\ltimes \R^n$ and ${\sf SO}(n)\ltimes \R^n$, for $n\geqslant 1$,
\item ${\sf Aut}(T_n)$, for $n=2,3,4,\ldots, \aleph_0$, where $T_n$ is the $n$-regular tree,
\item ${\sf Isom}(\Z\U)$, ${\sf Isom}(\Q\U)$, ${\sf Isom}(\U)$, where $\Z\U$, $\Q\U$ and $\U$ are respectively the integral, rational and full Urysohn metric spaces,
\item the group ${\sf Aff}(X)={\sf Isom}(X)\ltimes X$ of all affine isometries of an almost transitive Banach space $X$ whose group ${\sf Isom}(X)$ of linear isometries is globally bounded,
\item ${\sf Aff}(L^p[0,1])$  for all $1\leqslant p<\infty$.
\end{enumerate}
\end{restatable}

Although we are not able to fully characterise the class of asymptotically spherical groups, we do have a limited inverse to Theorem \ref{thm:main}. Namely, the asymptotic geodecity condition can be recovered from asymptotic sphericity itself. 

\begin{restatable}{thm}{main:inverse}
Suppose $\ell$ is a maximal length function on an asymptotically spherical group $G$. Then, for all $\eta>1$ and all sufficiently  large $r$, the following is valid for all $x$, 
$$
\inf\Big(\sum_{i=1}^k\ell(v_i)\;\Big|\; x=v_1\cdots v_k\;\;\&\;\; \ell(v_i)<r\Big)\;<\;\eta\cdot \ell(x).
$$
\end{restatable}

Alternatively, if $d$ is the left-invariant metric on $G$ associated with the maximal length function $\ell$, then the metric space $(G,d)$ satisfies the geodecity condition of Theorem \ref{thm:main}.

In particular in the context of discrete groups, asymptotic sphericity may be too strong and thus to have more examples it is useful to weaken it a bit. 

\begin{defn}
    We say that a metrisable topological group that admits unbounded maximal length functions has {\em bounded eccentricity} provided that there is a constant $\Theta$ so that, for all maximal length functions $\ell_1$, $\ell_2$ on $G$, we have 
$$
\ell_1-C\leqslant \lambda\ell_2\leqslant \Theta\ell_1+C
$$
for some appropriate constants $C$ and $\lambda$.
\end{defn}


We learned of this concept from Sebastian Hurtado, who was interested in the question of whether lattices in higher rank Lie groups would have bounded eccentricity. The following result is a proved by a simple variation of Theorem \ref{thm:main} and was also noted by Hurtado.   

\begin{restatable}{prop}{bounded}
The group $C_{4}\ltimes \Z^2$ has bounded eccentricity but fails to be asymptotically spherical.
\end{restatable}

Here $C_4$ is the cyclic group of order $4$ generated by the automorphism $\theta$ of $\Z^2$ given by rotation of angle $\frac \pi2$.

On the other hand, there are a plethora of examples of unbounded eccentric groups, namely, any direct product of two groups admitting unbounded maximal length functions. 

\begin{example}
    The group $\Z^2$ has unbounded eccentricity, as witnessed by the length functions $\ell_n(x,y) \defeq n|x| + |y|$.
\end{example}

The concepts of asymptotic sphericity and bounded eccentricity bear resemblance to Y. de Cornulier's property (PL) \cite{Cornulier}.  This was defined originally for locally compact groups, but see \cite[Section 3.8]{Rosendal} for the extension to all topological groups. Property (PL) is another rigidity property of length functions stating, in the case of a metrisable topological group $G$, that every compatible length function on the group is either bounded or coarsely proper. In fact, if $G$ is a Polish group with property (PL) and $A$ is a closed generating set for $G$, then the associated word length function $\rho_A$ is either bounded or is quasi-isometric to a maximal length function on $G$ \cite[Corollary 3.64]{Rosendal}. As it turns out, many of our examples do have property (PL). This includes ${\sf SO}(n)\ltimes \R^n$ for $n\geqslant 2$ \cite[Proposition 1.8]{Cornulier}, ${\sf Aut}(T_n)$ for $n=3,4,\ldots$ \cite[Corollary A.6.]{Valette} and for $n=\aleph_0$ \cite[Example 3.71]{Rosendal}, ${\sf Isom}(\Z\U)$, ${\sf Isom}(\Q\U)$, ${\sf Isom}(\U)$, and ${\sf Aff}(X)$ for an almost transitive Banach space with globally bounded linear isometry group \cite[Section 3.8]{Rosendal}.

\begin{restatable}{prop}{propertyPL}\label{prop:PL}
Suppose $A$ and $B$ are closed generating sets for an asymptotically spherical Polish group $G$ with property (PL). Then either $\rho_A$ is bounded, $\rho_B$ is bounded or, for any $\eta>\log4$, we have 
$$
\rho_A-C\;\leqslant\; \lambda\rho_B\;\leqslant\; \eta\cdot \rho_A+C
$$
for some constants $\lambda$ and $C$.
\end{restatable}

Although we suspect that this holds for all $\eta>1$, we have no proof to this effect. Also, given the significant overlap between groups with property (PL) and asymptotically spherical groups, one may wonder if every Polish group with a maximal length function and property (PL) is also asymptotically spherical or at least has bounded eccentricity. The reverse implication does not hold, since $\Z$ is asymptotically spherical but  fails property (PL).


\section{A pseudometric on maximal length functions}
Our first task is to put a pseudometric on the class of maximal length functions on a (topological) group $G$ that measures the asymptotic distortion between the various induced notions of balls and spheres.

\begin{lemma}\label{lem:pseudometric}
For any two  unbounded maximal length functions $\ell_1$ and $\ell_2$ on a (topo\-logical)\footnote{To stress our point that abstract groups fall within the scope of our results when viewed as discrete topological groups, we put parentheses around ``topological.''} group $G$, we have
\mathes{
\limsup_{x\to \infty} \; \frac{\ell_1(x)}{\ell_2(x)} 
&=\inf \set{K \in \R\;|\; \text{for some } C \in \R, \; \ell_1 \leqslant K \ell_2 + C}\\
&= \bigg(\liminf_{x\to \infty} \; \frac{\ell_2(x)}{\ell_1(x)}\bigg)\inv.
}
\end{lemma}

\begin{proof}
Because both $\ell_1$ and $\ell_2$ are maximal, we have that $\tfrac 1K\ell_1-C\leqslant \ell_2\leqslant K\ell_1+C$ for some appropriate constants $K$ and $C$ and hence that
$$
0\;< \;\limsup_{x\to \infty} \; \frac{\ell_1(x)}{\ell_2(x)}\;< \; \infty.
$$
It follows from this that 
$$
\bigg(\liminf_{x\to \infty} \; \frac{\ell_2(x)}{\ell_1(x)}\bigg)\inv\;=\;\limsup_{x\to \infty} \; \frac{\ell_1(x)}{\ell_2(x)} .
$$

Suppose now that $K > \limsup_{x\to \infty} \; \frac{\ell_1(x)}{\ell_2(x)}$.     Then there is some $C \in \R$ such that for all $x \in G$ with $\ell_1(x) \geqslant C$, we have 
$$
\frac{\ell_1(x)}{\ell_2(x)} < K,
$$
that is,  $\ell_1(x) < K \ell_2(x)$. It follows that, for all $x \in G$, either 
$$
\ell_1(x) < C \leqslant K\ell_2(x) + C,
$$
or
$$
\ell_1(x) < K \ell_2(x) \leqslant K \ell_2(x) + C.
$$

Assume instead that $K < \limsup_{x\to \infty} \; \frac{\ell_1(x)}{\ell_2(x)}$  and fix some  $\epsilon > 0$ so that $K + \epsilon < \limsup_{x\to \infty} \; \frac{\ell_1(x)}{\ell_2(x)}$. Then we may find a sequence of elements  $x_n \in G$ satisfying
$$
\ell_2(x_n) > \frac{n}{\epsilon}
\qquad\&\qquad 
\frac{\ell_1(x_n)}{\ell_2(x_n)} > K + \epsilon,
$$
whereby
$$
\ell_1(x_n) > (K + \epsilon)\ell_2(x_n) > K \ell_2(x_n) + n.
$$
Hence, for all $n \in \N$, $\ell_1 \not\leqslant K \ell_2 + n$.
\end{proof}

\begin{defn}
Let $G$ be a (topological) group admitting unbounded maximal length functions. We define $\alpha(\ell_1,\ell_2)$ as follows. 
\mathes{
\alpha(\ell_1, \ell_2) 
&\defeq 
\log \,\limsup_{x\to \infty} \; \frac{\ell_1(x)}{\ell_2(x)}\;\;+\;\;\log \,\limsup_{x\to \infty} \; \frac{\ell_2(x)}{\ell_1(x)}
\\
&=
\log \,\limsup_{x\to \infty} \; \frac{\ell_1(x)}{\ell_2(x)}\;\;-\;\;\log\,\liminf_{x\to \infty} \; \frac{\ell_1(x)}{\ell_2(x)}.
}   
\end{defn}

\begin{prop}\label{prop:pseudometric}
Let $G$ be a (topological) group admitting unbounded maximal length functions. Then 
$\alpha$
defines a pseudometric on the collection of all maximal length functions on $G$. In particular,
$$
\alpha(\ell_1,\ell_2)=0 \quad\text{if and only if}\quad  \text{the limit }\lim_{x\to \infty}\frac{\ell_1(x)}{\ell_2(x)} \text{ exists}.
$$
\end{prop}

\begin{proof}
Note first that, if $\ell_1$, $\ell_2$ and $\ell_3$ are maximal length functions on $G$, then 
\mathes{
\limsup_{x\to \infty} \; \frac{\ell_1(x)}{\ell_3(x)}
&= \limsup_{x\to \infty} \; \frac{\ell_1(x)}{\ell_2(x)}\cdot \frac{\ell_2(x)}{\ell_3(x)}\\
&\leqslant \limsup_{x\to \infty} \; \frac{\ell_1(x)}{\ell_2(x)}\cdot  \limsup_{x\to \infty}\;\frac{\ell_2(x)}{\ell_3(x)}.\\
}
Thus, for all $\ell_1$, $\ell_2$ and $\ell_3$, we have
\mathes{
\alpha(\ell_1,\ell_3)
&=\log \,\limsup_{x\to \infty} \; \frac{\ell_1(x)}{\ell_3(x)}\;\;+\;\;\log \,\limsup_{x\to \infty} \; \frac{\ell_3(x)}{\ell_1(x)}
\\
&\leqslant \log\bigg(\limsup_{x\to \infty} \; \frac{\ell_1(x)}{\ell_2(x)}\cdot \limsup_{x\to \infty} \; \frac{\ell_2(x)}{\ell_3(x)}\bigg)\;\;+\;\;\log\bigg(\limsup_{x\to \infty} \; \frac{\ell_3(x)}{\ell_2(x)}\cdot \limsup_{x\to \infty} \; \frac{\ell_2(x)}{\ell_1(x)}\bigg)\\
&=
\log \limsup_{x\to \infty} \; \frac{\ell_1(x)}{\ell_2(x)}
+
\log \limsup_{x\to \infty} \; \frac{\ell_2(x)}{\ell_3(x)}
+
\log \limsup_{x\to \infty} \; \frac{\ell_3(x)}{\ell_2(x)}
+
\log \limsup_{x\to \infty} \; \frac{\ell_2(x)}{\ell_1(x)}\\
&=\alpha(\ell_1,\ell_2)+\alpha(\ell_2,\ell_3).\\
}
So $\alpha$ satisfies the triangle inequality.

Observe also that 
\mathes{
\alpha(\ell_1, \ell_2) \;=\;
\log \,\limsup_{x\to \infty} \; \frac{\ell_1(x)}{\ell_2(x)}\;\;-\;\;\log\,\liminf_{x\to \infty} \; \frac{\ell_1(x)}{\ell_2(x)}\;\geqslant \; 0,
}
so $\alpha$ is a pseudometric. The same expression also shows that $\alpha(\ell_1,\ell_2)=0$ if and only if the limit $\lim_{x\to \infty}\frac{\ell_1(x)}{\ell_2(x)}$ exists.
\end{proof}

In the light of Lemma \ref{lem:pseudometric} and Proposition \ref{prop:pseudometric}, we may thus give the following alternative definitions of asymptotic sphericity and bounded eccentricity. 
\begin{defn}
Let $G$ be a (topological) group admitting unbounded maximal length functions. We say that $G$ has {\em bounded eccentricity} if and only if the collection of maximal length functions on $G$ has finite $\alpha$-diameter and that $G$ is {\em asymptotically spherical} in case it has $\alpha$-diameter $0$.
\end{defn}


\section{A criterion for asymptotic sphericity}\label{sec:spherical}

We now proceed to establish our main result, Theorem \ref{thm:main}, namely sufficient conditions for a group to be asymptotically spherical.

\main*

\begin{proof}
Fix some point $a\in X$ as in the statement of the theorem and define a pseudolength function $\tau$ by 
$$
\tau(g)=d(g(a),a).
$$
We will show that $\alpha(\ell,\tau)=0$ for any maximal length function $\ell$ on $G$, which thus proves that $G$ is asymptotically spherical.

To see this, suppose that a maximal length function $\ell$ and a number $\eta>1$ are given. We first choose some $R$ such that 
$$
\inf\bigg(\sum_{i=1}^nd(z_{i-1},z_{i})\;\bigg|\; z_0=x,\; z_n=y\;\&\; d(z_{i-1},z_i)<R\bigg)\;<\;\eta\cdot d(x,y)
$$
for all $x,y\in X$ and let
$$
V=\{h\in G\;|\; \tau(h)<2\sigma(R)\}.
$$
Set 
$$
\lambda=\limsup_{g\to \infty}\frac {\tau(g)}{\ell(g)},
$$
whereby  
$$
\tfrac 1\eta\cdot \tau-K<\lambda\ell
$$ 
for some constant $K$, whereas, for all $r$, there is some $g_r\in G$ with $\eta\tau(g_r)>\lambda\ell(g_r)+r$. 
Choose $r>{2\sigma(R)\eta}$ large enough so that
$$
{\sf rad}_\ell(V):=\sup\{\ell(f)\;|\; f\in V\}<\frac r{\lambda}
$$
and let $g=g_r$. Thus, 
$$
\ell(g)+{\sf rad}_\ell(V)<\frac{\eta\cdot\tau(g)}\lambda.
$$ 

Suppose now that $f\in G$ is an arbitrary element. Then there are $z_0,z_1,\ldots, z_n\in X$ with $z_0=a$, $z_n=f(a)$ and $d(z_{i-1},z_i)<R$ so that 
$$
\sum_{i=1}^nd(z_{i-1},z_{i})\;\leqslant\;\eta\cdot d(f(a),a)\;=\;\eta\cdot \tau(f).
$$
This means that there is a subsequence $w_0,w_1,\ldots, w_m$ of the $z_i$ with $w_0=z_0=a$ and $w_m=z_n=f(a)$ so that
$$
\tau(g) \;<\;d(w_{i-1},w_i)\;<\;\tau (g)+R
$$
for all $i=1,2,\ldots, m-1$, whereas $d(w_{m-1},w_m)<\tau(g)+R$. It follows that
$$
(m-1)\cdot \tau(g)\;\leqslant \;\sum_{i=1}^md(w_{i-1},w_{i})\;\leqslant \;\sum_{i=1}^nd(z_{i-1},z_{i})\;\leqslant\;\eta\cdot \tau(f)
$$
and so 
$$
m-1\leqslant\frac{\eta\cdot\tau(f)}{\tau(g)}.
$$ 
Also, by the assumptions of the theorem, we may find $h_1,\ldots, h_{m-1}$ so that
\mathes{
d_{\sf Hausdorff}\big(\{h_i(w_{i-1}),h_i(w_i)\},\{g(a),a\}\big)
&<\sigma\big( d(w_{i-1},w_i)-d(g(a),a)\big)\\
&=\sigma\big(d(w_{i-1},w_i)-\tau(g)\big)\\
&\leqslant\sigma(R)
}
for all $i=1,\ldots,m-1$.

\begin{claim}
For all $j=0,\ldots, m-1$, we have ${\sf dist}\big(w_j, \big(V\{g,g\inv\}\big)^j\cdot a\big)<\sigma(R)$.
\end{claim}

\begin{proof}[Proof of claim]
The proof of the claim is by induction on $j$. The case $j=0$ is trivial, because $w_0=a$. So suppose instead by induction that $d(w_j, h(a))<\sigma(R)$ for some $h\in \big(V\{g,g\inv\}\big)^j$. Because 
$$
d(g(a),a)\;=\;\tau(g)\;=\;\tau(g_r)\;>\;\frac r\eta\;>\;2\sigma(R),
$$
we have by the choice of $h_{j+1}$ that either
$$
d(h_{j+1}(w_j),a)<\sigma(R) \quad\&\quad d(h_{j+1}(w_{j+1}),g(a))<\sigma(R) 
$$
or that
$$
d(h_{j+1}(w_j),g(a))<\sigma(R) \quad\&\quad d(h_{j+1}(w_{j+1}),a)<\sigma(R).
$$
In the first case, we see that 
\mathes{
\tau(h_{j+1}h)
&=d(h_{j+1}h(a), a)\\
&\leqslant d(h_{j+1}h(a),h_{j+1}(w_j))+d(h_{j+1}(w_j) ,a)\\
&= d(h(a),(w_j))+d(h_{j+1}(w_j) ,a)\\
&<2\sigma(R),
}
whereby $h_{j+1}h\in V$ and $h_{j+1}\inv g\in hVg\subseteq  \big(V\{g,g\inv\}\big)^{j+1}$. Thus
\mathes{
{\sf dist}\big(w_{j+1}, \big(V\{g,g\inv\}\big)^{j+1}\cdot a\big)
\;\leqslant\; d\big(w_{j+1}, h_{j+1}\inv g (a)\big)\;=\;d(h_{j+1}(w_{j+1}),g(a))\;<\;\sigma(R).
}
In the second case, 
\mathes{
\tau(g\inv h_{j+1}h)
&=d(g\inv h_{j+1}h(a),a)\\
&=d(h_{j+1}h(a),g( a))\\
&\leqslant d(h_{j+1}h(a),h_{j+1}(w_j))+d(h_{j+1}(w_j) ,g(a))\\
&= d(h(a),w_j)+d(h_{j+1}(w_j) ,g(a))\\
&<2\sigma(R),
} 
so $g\inv h_{j+1}h\in V$, i.e., $h_{j+1}\inv \in hVg\inv\subseteq  \big(V\{g,g\inv\}\big)^{j+1}$. It follows again that
\mathes{
{\sf dist}\big(w_{j+1}, \big(V\{g,g\inv\}\big)^{j+1}\cdot a\big)
\;\leqslant\; d\big(w_{j+1}, h_{j+1}\inv  (a)\big)\;=\;d(h_{j+1}(w_{j+1}),a)\;<\;\sigma(R),
}
which proves our claim.
\end{proof}

By our claim we now see that there is some $h\in \big(V\{g,g\inv\}\big)^{m-1}$ so that $d(w_{m-1},h(a))<\sigma(R)$ and hence so that
$$
\tau(h\inv f)=d(f(a),h(a))=d(w_m,h(a))\leqslant d(w_m, w_{m-1})+d(w_{m-1},h(a))<\tau(g)+R+\sigma(R).
$$
It thus follows that
\mathes{
\lambda\ell(f)
&\leqslant \lambda\ell(h)+\lambda\ell(h\inv f)\\
&\leqslant \lambda(m-1)\Big(\ell(g)+{\sf rad}_\ell(V)\Big)+\sup\{\lambda\ell(q)\;|\; \tau(q)<\tau(g)+R+\sigma(R)\}\\
&<  \lambda\cdot \frac{\eta\cdot\tau(f)}{\tau(g)}\cdot \frac{\eta\cdot\tau(g)}\lambda+\sup\{\lambda\ell(q)\;|\; \tau(q)<\tau(g)+R+\sigma(R)\}\\
&= {\eta^2\cdot\tau(f)}+\sup\{\lambda\ell(q)\;|\; \tau(q)<\tau(g)+R+\sigma(R)\}.
}
Set $K'=\max\Big\{K, \sup\{\lambda\ell(q)\;|\; \tau(q)<\tau(g)+R+\sigma(R)\} \Big\}$. Because $f\in G$ was arbitrary, we thus see that
$$
\tfrac1\eta\cdot\tau-K'\;\;<\;\;\lambda\ell\;\;<\;\;{\eta^2}\cdot\tau+K'.
$$
Again, as $\eta>1$ is arbitrary, we find that $\alpha(\ell,\tau)=\alpha(\lambda\ell, \tau)=0$.
\end{proof}

We may now combine Theorem \ref{thm:main} with the general version of the Milnor--Schwarz lemma \cite[Theorem 2.77]{Rosendal} to get a more immediately applicable result. For this we need a few additional concepts. First of all, a subset $B$ of a metrisable topological group $G$ is said to be {\em bounded} in case 
$$
{\sf rad}_\ell(B)<\infty
$$
for every compatible length function $\ell$ on $G$. In case $G$ admits a maximal length function $\ell$, it is evidently enough to check that ${\sf rad}_\ell(B)<\infty$ for this specific $\ell$. Also, an isometric action of a group $G$ on a metric space $X$ is said to be {\em cobounded} provided that, for any $x\in X$, 
$$
\sup_{y\in X}\inf_{g\in G}d(y,gx)<\infty.
$$
And finally, the action is {\em coarsely proper} provided that, for all subsets $B\subseteq G$ and $x\in X$,
$$
B \text{ is bounded } \quad\Longleftrightarrow \quad {\sf diam}(B\cdot x)<\infty.
$$

\begin{cor}\label{cor:coarsely proper}
Suppose $G\curvearrowright X$ is a coarsely proper, cobounded, continuous, isometric action of a metrisable topological group on a metric space. Assume also that 
\mathes{
\a \eta>1\;\;\exists R\;\; \a x,y\in X\;\;\inf\bigg(\sum_{i=1}^nd(z_{i-1},z_{i})\;\bigg|\; z_0=x,\; z_n=y\;\&\; d(z_{i-1},z_i)<R\bigg)<\eta\cdot d(x,y)
}
and that, for some increasing function $\sigma\colon \R_+\to \R_+$ and all $x,y,z,u\in  X$ with $d(x,y)\leqslant d(z,u)$, we have 
$$
\inf_{h\in G}d_{\sf Hausdorff}\big(\{h(z),h(u)\},\{x,y\}\big)< \sigma\big(d(z,u)-d(x,y)\big).
$$
Then $G$ is asymptotically spherical.
\end{cor}

\begin{proof}
The metric space $X$ is evidently large-scale geodesic in the sense of \cite[Definition 2.62]{Rosendal}. Therefore, by \cite[Theorem 2.77]{Rosendal} the orbital map
$$
g\in G\mapsto gx\in X
$$
is a quasi-isometric embedding.
\end{proof}

All cases of Corollary \ref{cor:main} are now all special instances of Corollary \ref{cor:coarsely proper}. Indeed, for $n\geqslant 1$, the groups ${\sf O}(n)\ltimes \R^n$  and ${\sf SO}(n)\ltimes \R^n$ consist of all isometries, respectively, orientation-preserving isometries of $\R^n$. Their tautological actions on $\R^n$ are clearly transitive and proper, whereby they are also coarsely proper. Furthermore, the $2$-transitivity condition is evident. Similarly, for $n\geqslant 2$, the tautological action of  ${\sf Aut}(T_n)$ on the metric space $T_n$ is proper and transitive. The case of ${\sf Aut}(T_\infty)$ similarly follows from \cite[Example 3.6]{Rosendal}, whereas ${\sf Isom}(\Z\mathbb U)$, ${\sf Isom}(\Q\mathbb U)$ and ${\sf Isom}(\mathbb U)$ follow from \cite[Example 6.36 and Theorem 3.10]{Rosendal}, respectively. Finally, observe that $\Z$ can be seen as the group of orientation-preserving isometries of the bi-infinite line $T_2$, whereas 
the infinite dihedral group $D_\infty$ is just ${\sf Aut}(T_2)$ and $\R$ is isomorphic to ${\sf O(1)}\ltimes \R$.

For the case of Banach spaces, note that, by \cite[Proposition 3.17]{Rosendal}, if ${\sf Isom}(X)$ is globally bounded, then the orbital map $g\in {\sf Aff}(X)\mapsto g\cdot 0\in X$ is a quasi-isometric embedding. Also, if $X$ is {\em almost transitive}, meaning that the group ${\sf Isom}(X)$ of linear isometries of $X$ induces dense orbits on the unit sphere $S_X$, then the $2$-transitivity condition of Corollary \ref{cor:coarsely proper}  for the point-transitive action ${\sf Aff}(X)\curvearrowright X$ is easily satisfied. Indeed, in this case, for any points $x,y,z,u\in X$ satisfying $0<\norm{y-x}\leqslant \norm{u-z}$ and any $\epsilon>0$, one may may find a linear isometry $T\in {\sf Isom}(X)$ so that 
$$
\NORMM{T\Big(\frac {u-z}{\norm{u-z}}\Big)-\frac {y-x}{\norm{y-x}}}<\frac\epsilon{\norm{u-z}},
$$
whereby $h=T(\,\cdot\, -z)+x$ defines an affine isometry of $X$ satisfying $h(z)=x$ and 
\mathes{
\norm{h(u)-y}&=
\norm{T(u-z)+x-y}\\
&= \norm{u-z}\NORMM{T\Big(\frac {u-z}{\norm{u-z}}\Big)-\frac {y-x}{\norm{u-z}}}\\
&< \norm{u-z}\NORMM{\frac {y-x}{\norm{y-x}}-\frac {y-x}{\norm{u-z}}} +\epsilon\\
&= \NORMM{\big(\norm{u-z}-\norm{y-x}\big)\frac {y-x}{\norm{y-x}}} +\epsilon\\
&=  \big|\norm{u-z}-\norm{y-x}\big|+\epsilon.
}
Thus, $d_{\sf Hausdorff}\big(\{h(z),h(u)\},\{x,y\}\big)<\big|\norm{u-z}-\norm{y-x}\big|+\epsilon$. On the other hand, when $\norm{y-x}=0$, i.e., $x=y$, then one may simple let $h$ be the translation by $x-z$, whence $d_{\sf Hausdorff}\big(\{h(z),h(u)\},\{x,y\}\big)=\norm{u-z}=\big|\norm{u-z}-\norm{y-x}\big|$.

We may now apply this result to the spaces $X=L^p([0,1])$, $1\leqslant p<\infty$, whose almost transitivity was established by A. Pe\l czy\'nski and S. Rolewicz \cite[Theorem 9.6.3, Theorem 9.6.4]{Rolewicz}. That the linear isometry groups ${\sf Isom}\big(L^p([0,1])\big)$ are globally bounded is, in turn, a consequence of results due to C. W. Henson \cite{Henson}. Indeed, by \cite[Fact 17.6]{Henson}, for $1\leqslant p<\infty$, the Banach lattice $L^p([0,1])$ is $\omega$-categorical in the sense of model theory for metric structures. That means that its theory has a unique separable model. Furthermore, by \cite[Theorem 12.10 and Corollary 12.11]{Henson}, the tautological action ${\sf Isom}\big(L^p([0,1])\big)\curvearrowright S_{L^p([0,1])}$ on the unit sphere of ${L^p([0,1])}$ is approximately oligomorphic, which by \cite[Theorem 5.2]{OB} implies that ${\sf Isom}\big(L^p([0,1])\big)$ is globally bounded.


\section{Discrete groups of  bounded eccentricity}\label{sec:eccentric}
The examples of Section \ref{sec:spherical} motivate two questions that we will answer by a single construction. Namely, is there any difference between asymptotic sphericity and bounded eccentricity? And are the only finitely generated discrete groups of bounded eccentricity those of Corollary \ref{cor:main}, that is, the virtually cyclic groups? 

For the discussion here and in the next section, the following notation is useful.
\begin{defn}
For two functions $\phi,\psi\colon G\to \R$, set
    $$
    \phi\lesssim \psi\quad\equi\quad \limsup_{x\to \infty}\frac {\phi(x)}{\psi(x)}\leqslant 1\quad\equi\quad \liminf_{x\to \infty}\frac {\psi(x)}{\phi(x)}\geqslant 1,
    $$
    and set
    $$
    \phi\sim \psi \quad\equi\quad \phi\lesssim \psi\;\;\&\;\; \psi\lesssim \phi \quad\equi\quad \lim_{x\to \infty}\frac {\psi(x)}{\phi(x)}=1.
    $$
\end{defn}

Let $e_1, e_2$ denote the free basis of the abelian group $\Z^2$. The elements of $\Z^2$ can thus be expressed  as 
$$
n_1e_1+n_2e_2
$$
for  unique $n_i\in \Z$. Let also $\theta$ be the automorphism of $\Z^2$ of order $4$ given by
$$
e_1\;\overset{\theta}{\longmapsto}\;e_2\;\overset{\theta}{\longmapsto}\;-e_1\;\overset{\theta}{\longmapsto}\;-e_2\;\overset{\theta}{\longmapsto}\;e_1.
$$
The automorphism $\theta$ thus induces the semidirect product
$$
C_{4}\ltimes \Z^2,
$$
where $C_{4}$ denotes the cyclic group of order $4$ with generator $\theta$. Keeping the additive notation for the normal subgroup $\Z^2$ and multiplicative notation in the semidirect product, we thus have the basic relation
$$
\theta\cdot(n_1e_1+n_2e_2)\cdot \theta\inv=-n_2e_1+n_1e_2.
$$

\bounded*


\begin{proof}
Fix a maximal length function $\ell$ on $C_{4}\ltimes \Z^2$. By rescaling $\ell$, we may assume that
$$
\lim_{n\to \infty}\frac{\ell(ne_1)}{n}=1.
$$
For all $n\in \Z$, we have $\theta\cdot ne_1\cdot \theta\inv=e_{2}$ and hence $\big|\ell(ne_1)-\ell(ne_{2})\big| \leqslant 2\cdot \ell(\theta)$. 
It thus follows that
$$
\lim_{n\to \infty}\frac{\ell(-ne_i)}{n}=\lim_{n\to \infty}\frac{\ell(ne_i)}{n}=\lim_{n\to \infty}\frac{\ell(ne_1)}{n}=1,
$$
that is, 
$$
\lim_{|n|\to \infty}\frac{\ell(ne_i)}{|n|}=1.
$$
Observe also that 
$$
2ne_2=(ne_1+ne_2)+(-ne_1+ne_2)=(ne_1+ne_2)+\theta\cdot(ne_1+ne_2)\cdot\theta\inv,
$$
whereby
$$
\ell(2ne_2)\leqslant 2\ell(ne_1+ne_2)+2\ell(\theta).
$$

\begin{claim}
$$
\limsup_{|n_1|+|n_2|\to \infty}\frac{\ell(n_1e_1+n_2e_2)}{|n_1|+|n_2|}\leqslant1.
$$
\end{claim}

\begin{proof}
Let $\eta>1$ be given.  We pick $N_1$ large enough so that
${\ell(ne_i)}<\eta\cdot{|n|}$
whenever  $|n|\geqslant N_1$. Then, if $n_1,n_2\in \Z$ with 
$$
|n_1|+|n_2|\geqslant \frac {\max\big\{\ell(n_1e_1+n_2e_2)\;\big|\;  |n_i|<N_1\big\} }{\eta(\eta-1)},
$$
we have  
\mathes{
\ell(n_1e_1+n_2e_2)
&\leqslant \ell\bigg(\sum_{|n_i|<N_1} n_ie_i     \bigg)
\;+\;
\sum_{|n_i|\geqslant N_1} \ell(n_ie_i)   \\
&\leqslant\eta(\eta-1)\Big(|n_1|+|n_2|\Big)+ \sum_{|n_i|\geqslant N_1}\eta\cdot |n_i|\\
&\leqslant\eta(\eta-1)\Big(|n_1|+|n_2|\Big)+\eta\cdot \Big(|n_1|+|n_2|\Big)\\
&\leqslant\eta^2\cdot\Big( |n_1|+|n_2|\Big).
\qedhere
}
\end{proof}
Observe that for all $0\leqslant n_1\leqslant n_2$,
$$
\ell(n_1e_1+n_2e_2)\geqslant \ell(n_2e_2)-\ell(-n_1e_1)= \ell(n_2e_2)-\ell(n_1e_1),
$$
and
\mathes{
\ell(n_1e_1+n_2e_2)
&\geqslant \ell(n_2e_1+n_2e_2)-\ell\big((n_2-n_1)e_1\big)\\
&\geqslant \tfrac {\ell(2n_2e_2)}2-\ell(\theta)-\ell\big((n_2-n_1)e_1\big).
}
Moreover, either  $n_1\leqslant\tfrac {n_2}2$ or $n_2-n_1\leqslant\tfrac {n_2}2$.
It follows from this that
$$
\liminf_{\substack{n_1\to \infty   \\  0\leqslant n_1\leqslant n_2}}\frac{\ell(n_1e_1+n_2e_2)}{n_2}\geqslant \frac 12.
$$
Treating other cases similarly, we find that
$$
\frac 12\;\;\leqslant\;\; \liminf_{|n_1|+|n_2|\to \infty}\frac{\ell(n_1e_1+n_2e_2)}{\max\big\{|n_1|,|n_2|\big\}}\;\;\leqslant\;\;
\limsup_{|n_1|+|n_2|\to \infty}\frac{\ell(n_1e_1+n_2e_2)}{|n_1|+|n_2|}\;\;\leqslant\;\;1.
$$
Thus, if we define two pseudolength functions $\norm\cdot_1$ and $\norm\cdot_\infty$ on $C_4\ltimes \Z^2$ by
$$
\Norm{\theta^i\cdot(n_1e_1+n_2e_2)}_1=|n_1|+|n_2|\quad\&\quad \Norm{\theta^i\cdot(n_1e_1+n_2e_2)}_\infty=\max\big\{|n_1|,|n_2|\big\},
$$
we have 
$$
\tfrac 14\norm\cdot_1\;\;\lesssim\;\; \tfrac 12\norm\cdot_\infty \;\;\lesssim\;\;  \ell\;\;\lesssim\;\;  \norm\cdot_1
$$
and conclude that $C_4\ltimes \Z^2$ has bounded eccentricity.

On the other hand, $\alpha(\norm\cdot_1, \norm\cdot_\infty) = \log(2)$, so $C_4\ltimes \Z^2$ is not asymptotically spherical.
\end{proof}

We can see from the above proof that, for any two maximal length functions $\ell$ and $\ell'$ on $C_{4}\ltimes \Z^2$, $\alpha(\ell,\ell') \leqslant 4\log(2)$; meanwhile, the $\alpha$-diameter of $C_{4}\ltimes \Z^2$ is at least $\log(2)$. 
However, we do not know the exact $\alpha$-diameter.

\begin{question}
    What is the $\alpha$-diameter of $C_{4}\ltimes \Z^2$?
\end{question}

\begin{remark}
    By a similar argument, we see that $C_{4}\ltimes \R^2$ (where $C_4$ acts by rotation of $\R^2$ as in the example above) also has bounded eccentricity, but is not asymptotically spherical.
    Thus, in both the discrete and non-discrete cases, bounded eccentricity does not imply asymptotic sphericity.
\end{remark}


\section{Geodesic properties of asymptotically spherical groups}

If $A$ is any generating set for a group $G$, we let $\rho_A$ denote the {\em word length function} associated with $A$, that is,
$$
\rho_A(x)=\inf\big(k\;\big|\; x=a_1\cdots a_k\;\;\&\;\; a_i\in A^\pm\big).
$$
Of course, if $G$ is non-discrete, the integral valued length function $\rho_A$ will not be compatible with the topology on $G$, although depending on $A$ it may still be quasi-isometric with a maximal length function.  Thus, if $G$ is a metrisable topological group admitting an unbounded length function $\ell$, we extend the pseudometric $\alpha$ to the class of all length functions on $G$ that are quasi-isometric to $\ell$ by the formula of Proposition \ref{prop:pseudometric}.

    Suppose $G$ is a metrisable topological group admitting  unbounded length functions and set
    \begin{enumerate}
        \item $\mathcal M=$ the set of all compatible  maximal length function on $G$,
        \item $\mathcal L=$ the set of all length function on $G$ quasi-isometric to maximal length functions,
        \item $\mathcal W=$ the set of all word length functions $\rho_A$ on $G$ quasi-isometric to maximal length functions.
    \end{enumerate}

\begin{problem}
How are the three quantities ${\sf diam}_\alpha(\mathcal M)$, ${\sf diam}_\alpha(\mathcal L)$ and ${\sf diam}_\alpha(\mathcal W)$ related? In particular, do the following equivalences hold,
$$
{\sf diam}_\alpha(\mathcal M)=0\quad\equi\quad {\sf diam}_\alpha(\mathcal L)=0\quad\equi\quad {\sf diam}_\alpha(\mathcal W)=0\;?
$$
Similarly, what is the maximal $\alpha$-distance from an element of $\mathcal W$ to the set $\mathcal M$ and vice versa? 
\end{problem}

We shall not discuss this problem in full generality, but instead concentrate on the better behaved class of Polish groups. 
So let $G$ be a fixed Polish group and recall that a subset $A\subseteq G$ is bounded if and only if ${\sf rad}_\ell(A)<\infty$ for all compatible length functions $\ell$ on $G$. Note also that $A$ is bounded if and only if the topological closure $\overline A$ is bounded. 
We also  recall a few results from \cite[pp. 60-63]{Rosendal}. First of all, $G$ admits a maximal length function if and only if it is generated by a bounded set (which we may take to be closed or open).  Secondly, for any  bounded generating Borel (or even analytic) set $A$, the associated word length function $\rho_A$ is quasi-isometric to a maximal length function on $G$. Thirdly, a compatible length function $\ell$ is maximal if and only  if is is quasi-isometric to the word length function $\rho_A$ for a bounded generating set $A$. For these reasons, it makes sense to restrict the attention to word length functions $\rho_A$, where $A$ is either closed or open, rather than arbitrary bounded sets $A$.

The following construction (cf. \cite[Lemma 2.70]{Rosendal}) will be come into play at several occasions.
\begin{lemma}\label{lem:smoothed length functions}
    Suppose $\ell$ is a compatible length function on a Polish group $G$ and that $A\subseteq G$ is a symmetric generating set so that $B_\ell(r)\subseteq A\subseteq B_\ell(R)$ for some radii $0<r\leqslant R$. Then
    $$
    \ell_A(x)=\inf\Big(\sum_{i=1}^k\ell(a_i)\;\Big|\; x=a_1\cdots a_k\;\;\&\;\; a_i\in A\Big)
    $$
    defines a compatible length function on $G$ satisfying
    $$
    \tfrac r2\cdot {\rho_A} - \tfrac r2\;\leqslant\; \ell_A\:\leqslant\; R\cdot \rho_A.
    $$
\end{lemma}

\begin{proof}
Evidently $\ell_A$ is a length function satisfying $\ell\leqslant \ell_A$. Moreover, for all $x\in B_\ell(r)$, we have $\ell_A(x)=\ell(x)$, which implies that $\ell_A$ is compatible with the topology on $G$. 

Observe that, if  $x=a_1\cdots a_k$ is any factorisation of $x$ as a product of elements $a_i\in A$ and $a_ia_{i+1}\in A$ for some $i$, then 
$$
\ell_A(x)\;\leqslant\; \ell(a_1)+\ldots+\ell(a_{i-1})+\ell(a_ia_{i+1})+\ell(a_{i+2})+\cdots+\ell(a_k)\;\leqslant\; \sum_{i=1}^k\ell(a_i).
$$
This means that, in the definition of $\ell_A$, we may assume that $a_ia_{i+1}\notin A$ and hence that $\ell(a_ia_{i+1})\geqslant r$ for all $i$. Observe also that, if $x=a_1\cdots a_k$ is such a factorisation with say $k$ odd, then
$$
r\cdot \frac {{\rho_A}(x) - 1}2\;\;\leqslant\;\;r\cdot \frac {k- 1}2\;\;\leqslant\;\; \ell(a_1a_2)+\cdots+ \ell(a_{k-2}a_{k-1})\;\;\leqslant\;\; \sum_{i=1}^k\ell(a_i).
$$
Similarly if $k$ is even. Therefore, $r\cdot \frac {{\rho_A} - 1}2\;\leqslant\; \ell_A$. Conversely, for any factorisation $x=a_1\cdots a_k$ into elements of $A$, we have $\sum_{i=1}^k\ell(a_i)\leqslant R\cdot k$, which shows that $\ell_A\leqslant R\cdot \rho_A$. We conclude that
$$
r\cdot \frac {{\rho_A} - 1}2\;\leqslant\; \ell_A\:\leqslant\; R\cdot \rho_A
$$
and hence that $\rho_A$ and $\ell_A$ are quasi-isometric.
\end{proof}

\begin{theorem}\label{thm:smoothed length}
Suppose $\ell$ is a compatible maximal length function on an asymptotically spherical group $G$. Then, for all $\eta>1$ and all sufficiently large $R$, we have 
$$
\ell\;\leqslant\; \ell_{B_\ell(R)}\;<\; \eta\cdot \ell,
$$
i.e., every  $x\in G$ can be factored as $x=v_1v_2\cdots v_k$ where $\ell(v_i)<R$ and 
$$
\sum_{i=1}^k\ell(v_i)<\eta\cdot\ell(x).
$$
\end{theorem}

\begin{proof}
    For every symmetric open and bounded generating set $1\in V\subseteq G$,  $\ell_V$ is a compatible length function on $G$ quasi-isometric to the word length function $\rho_V$. So $\ell_V$ is maximal and 
 we may therefore define
    $$
    \beta_V\defeq \lim_{x\to \infty}\frac{\ell_V(x)}{\ell(x)}.
    $$
    Observe also that,  if $W$ is another symmetric open and bounded generating set with $1\in V\subseteq W\subseteq G$, then $\ell\leqslant \ell_W\leqslant \ell_V$, which shows that 
    $$
    1\leqslant \beta_W\leqslant \beta_V.
    $$
    We thus let $\beta=\inf_V\beta_V$ and claim that $\beta=1$.

    Note that, if the claim has been established and $\eta>1$ is given, then we may find some symmetric open and bounded generating set $1\in V\subseteq G$ so that $\beta_V<\eta$. Let also $R$ be large enough so that $V\subseteq B_\ell(R)$ and so that, for all $x\in G$ with $\ell(x)\geqslant R$, we have 
    $$
    \frac{\ell_V(x)}{\ell(x)}<\eta.
    $$
    Then, for such $x$,  we can write $x=v_1\cdots v_k$ with $v_i\in V$, whereby $\ell(v_i)<R$, and 
    $$
    \sum_{i=1}^k\ell(v_i)<\eta\cdot\ell(x).
    $$
    On the other hand, if $\ell(x)<R$, then $x=v_1$ is its own factorisation. This shows that the theorem follows from the claim.

    For the proof of the claim, suppose for a contradiction that $\beta>1$ and choose some $\eta>1$ so that $\eta^4<\beta$. Let also $r_1$ be large enough such that
    $$
    \beta\leqslant \beta_{B_\ell(r_1)}<\eta\beta
    $$
    and some $r_2>r_1$ such that
    \begin{equation}\label{eq:r1}
    \frac \beta\eta\ell(x)<\ell_{B_\ell(r_1)}(x)<\eta\beta\cdot\ell(x)
    \end{equation}
    whenever $\ell(x)\geqslant r_2$. We then pick $r_3\geqslant r_2$ so that $1+\tfrac{r_2}{r_3}<\eta$ and also
      \begin{equation}\label{eq:r2}
    \frac \beta\eta\ell(x)<\ell_{B_\ell(r_2+r_1)}(x)<\eta\beta\cdot\ell(x)
    \end{equation}
    whenever $\ell(x)\geqslant r_3$.

    Choose now any $y\in G$ with $\ell(y)>r_3+r_2$. By \eqref{eq:r1}, we may then write $y=v_1\cdots v_k$ for some $v_i\in B_\ell(r_1)$ so that
     \begin{equation}\label{eq:vi}
    \sum_{i=1}^k\ell(v_i)<\eta\beta\cdot\ell(y).
    \end{equation}
    Grouping from the left, we can then find $0=n_1<n_2<\ldots< n_p\leqslant k$ so that
     \begin{equation}\label{eq:vnj}
    r_2\leqslant \ell(v_{n_j+1}v_{n_j+2}\cdots v_{n_{j+1}})<r_2+r_1
    \end{equation}
    for all $j=1,\ldots, p-1$, whereas 
    \begin{equation}\label{eq:vnjr2}
    \ell(v_{n_p+1}v_{n_p+2}\cdots v_{k})<r_2.
    \end{equation}
    Letting $y'=v_1v_2\cdots v_{n_p}$, we see that 
    \begin{equation}\label{eq:y'r3}
    \ell(y')\geqslant \ell(y)-\ell(v_{n_p+1}v_{n_p+2}\cdots v_{k})>(r_3+r_2)-r_2=r_3
     \end{equation}
    and therefore 
    \begin{equation}\label{eq:y'}
    \frac\beta\eta\ell(y')\;<\;\ell_{B_\ell(r_2+r_1)}(y')\;\leqslant\;\sum_{j=1}^{p-1}\ell(v_{n_j+1}v_{n_j+2}\cdots v_{n_{j+1}})
    \end{equation}
    by \eqref{eq:r2}. Also, 
    \mathes{
    \sum_{i=1}^{n_p}\ell(v_i)
    &\leqslant\sum_{i=1}^k\ell(v_i)\\
    &\stackrel{\eqref{eq:vi}}{<} \eta\beta\cdot\ell(y)\\
    &\leqslant\eta\beta\cdot\big(\ell(y')+\ell(v_{n_p+1}v_{n_p+2}\cdots v_{k})\big)\\
    &\stackrel{\eqref{eq:vnjr2}}{<} \eta\beta\cdot\big(\ell(y')+r_2\big)\\
    &\stackrel{\eqref{eq:y'r3}}{<} \eta\beta\cdot\big(\ell(y')+\tfrac{r_2}{r_3}\ell(y')\big)\\
    &=\eta\beta\big(1+\tfrac{r_2}{r_3}\big)\cdot\ell(y')\\
    &<\eta^2\beta\cdot\ell(y').\\
    }
    On the other hand, by \eqref{eq:vnj} and \eqref{eq:r1}, we see that, for all $j=1,\ldots, p-1$,
    \begin{equation}\label{eq:j}
    \frac \beta\eta\ell(v_{n_j+1}v_{n_j+2}\cdots v_{n_{j+1}})<\ell_{B_\ell(r_1)}(v_{n_j+1}v_{n_j+2}\cdots v_{n_{j+1}})\leqslant \sum_{i=n_j+1}^{n_{j+1}}\ell(v_i),
    \end{equation}
    whereby
    \mathes{
    \sum_{i=1}^{n_p}\ell(v_i)
    &<\eta^2\beta\cdot\ell(y')\\
    &\stackrel{\eqref{eq:y'}}{< } \eta^3\cdot\sum_{j=1}^{p-1}\ell(v_{n_j+1}v_{n_j+2}\cdots v_{n_{j+1}})\\
    &\stackrel{\eqref{eq:j}}{< } \frac{\eta^4}\beta\cdot\sum_{i=1}^{n_p}\ell(v_i),\\
    }
    which is absurd as $\frac{\eta^4}\beta<1$. This contradiction proves our claim and thus the theorem.
\end{proof}

\begin{lemma}\label{smoothed vs rho}
    Suppose $\ell$ is a length function on a Polish group $G$ and that $G$ is generated by the ball $B_\ell(r)$ of some radius $r>0$. Then, for all $\zeta< 1$ and all sufficiently large $R$, we have
$$
\zeta\cdot R\rho_{B_\ell(R)}\;\;\lesssim\;\; \ell_{B_\ell(r)}\;\;\leqslant\;\; r\rho_{B_\ell(r)} .
$$
\end{lemma}

\begin{proof}
Suppose that $x=v_1\cdots v_k$ with $\ell(v_i)<r$. Let $R$ large enough so that $\zeta<\frac {R-r}R$. By grouping from the left, we can find $0=m_0<m_1<\ldots<m_p\leqslant k$ such that
$$
R-r\;\leqslant\; \ell(v_{m_j+1}\cdots v_{m_{j+1}})\;<\;R
$$
for all $j=0,\ldots, p-1$ and 
 $$
\ell(v_{m_p+1}v_{m_p+2}\cdots v_k)<R.
$$   
We set $w_j=v_{m_j+1}\cdots v_{m_{j+1}}$ for $j=0,\ldots, p-1$ and $w_p=v_{m_p+1}v_{n_p+2}\cdots v_k$, whereby $x=w_0w_1\cdots w_p\in B_\ell(R)^{p+1}$ and $\rho_{B_\ell(R)}(x)\leqslant p+1$.  Furthermore,
\mathes{
\Big(\rho_{B_\ell(R)}-1\Big)(R-r)    \;\leqslant\; p(R-r)\;\leqslant \;\sum_{j=0}^p\ell(w_j)\;\leqslant\;\sum_{i=1}^k\ell(v_i),
}
which shows that $\frac {R-r}R\cdot R\rho_{B_\ell(R)}\lesssim\ell_{B_\ell(r)}$. That $\ell_{B_\ell(r)}\leqslant r\rho_{B_\ell(r)}$ is evident and therefore
$$
\zeta\cdot R\rho_{B_\ell(R)}\;\;\lesssim\;\;\ell_{B_\ell(r)}\;\;\leqslant\;\; r\rho_{B_\ell(r)}$$
as was claimed.
\end{proof}

\begin{thm}\label{theorem_word_length}
  Suppose $\ell$ is a compatible maximal length function on an asymptotically spherical group $G$. Then, for every $\eta>1$ and all sufficiently large  $r$, we have
  $$
\ell\;\leqslant\;  r\rho_{B_\ell(r)}\;\lesssim\; \eta\cdot \ell.
  $$
  In particular, 
  $$
  \lim_{r\to \infty}\alpha\big( \rho_{B_\ell(r)}, \ell\big)=0.
  $$
\end{thm}

\begin{proof}
    Let $\eta>1$ be given and find first, by Theorem \ref{thm:smoothed length}, some large enough $R$ so that $B_\ell(R)$ generates $G$ and so that $\ell_{B_\ell(R)}<\eta^{\frac 12}\cdot \ell$. By Lemma \ref{smoothed vs rho}, it now follows that, for all sufficiently large $r$,
    $$
     \eta^{-\frac 12}\cdot r\rho_{B_\ell(r)}\;\;\lesssim\;\; \ell_{B_\ell(r)}\;\;<\;\;\eta^{\frac 12}\cdot \ell,
    $$
    i.e., $r\rho_{B_\ell(r)}\;\;\lesssim\;\; \eta\cdot \ell$. That $\ell\leqslant r\rho_{B_\ell(r)} $ is evident.
\end{proof}

It is not clear to us if this result extends to word length functions with respect any closed bounded generating sets $A$. Nor do we know if it only holds asymptotically. Specifically, the following questions remains open.

\begin{question}
    Can we take $\eta = 1$ in the statement of \cref{theorem_word_length}?
    That is, is it true that for sufficiently large $r$, $\alpha(\rho_{B_\ell(r)},\ell) =0$?
\end{question}

\begin{question}
    Suppose $A$ is a closed bounded generating set for an asymptotically spherical Polish group $G$ and $\ell$ is a  maximal length function on $G$. Does it follow that 
    $$
    \alpha(\rho_A,\ell)=0?
    $$
\end{question}

The following result gives the best bound known to us.
\begin{prop}
    Suppose $A$ is a closed bounded  generating set for a Polish group $G$. Then there is a maximal length function $\ell$ on $G$ for which
    $$
    \alpha(\rho_A,\ell)\leqslant \log 2.
    $$
\end{prop}

\begin{proof}
By replacing $A$ with the set $A\cup A\inv \cup \{1\}$, we may without loss of generality assume that $A$ is symmetric and contains the identity element $1$. Also, because $A$ is closed, $A^1\subseteq A^2\subseteq A^3\subseteq \ldots \subseteq G$ is an increasing exhaustive sequence of analytic subsets of $G$ and thus, by the Baire category theorem, some $A^p$ must be non-meagre. By Pettis' theorem \cite[Theorem 9.9]{Kechris}, it follows that $C=A^{2p}$ is a symmetric identity neighbourhood. 

Fix a compatible maximal metric $l$ on $G$ and some $r>0$ so that $B_l(r)\subseteq C$.
Define a new compatible length function $l'$ on $G$ by letting
    $$
    l'(x)=\min\{l(x),r\}
    $$
    and note that $B_{l'}(r)\subseteq C\subseteq B_{l'}(R)$ for all $R>r$.
By Lemma \ref{lem:smoothed length functions}, the associated length function $l'_C$ is compatible and quasi-isometric to the word metric $\rho_C$, in fact, $\alpha(\rho_C,l'_C)\leqslant 2$. So $l'_C$ is  maximal. As $\alpha(\rho_A,\rho_C)=0$, the result follows.
\end{proof}

Using this, Proposition \ref{prop:PL} follows easily.
\propertyPL*

\begin{proof}
Suppose $A$ and $B$ are closed bounded generating sets for $G$ and let $l$ and $l'$ be maximal length functions on $G$ so that $
\alpha(\rho_A,\ell)\leqslant \log 2$ and $\alpha(\rho_B,\ell')\leqslant \log 2$. Because $G$ is asymptotically spherical, we find that $\alpha(\rho_A,\rho_B)\leqslant \log 2+\log 2=\log 4$, which implies the statement of the proposition.
\end{proof}

\begin{prop}
    Let $\ell$ be a maximal length function on a Polish group $G$ and suppose that
    $$
    \lim_{r_1,r_2\to \infty}\alpha\big(\rho_{B_\ell(r_1)},\rho_{B_\ell(r_2)}\big)=0.
    $$
    Then    
    $$
    \lim_{r_1,r_2\to \infty}\alpha\big(\ell_{B_\ell(r_1)},\rho_{B_\ell(r_2)}\big)=0.
    $$
\end{prop}

\begin{proof}
Suppose for a contradiction that we have
$$
\epsilon:=\limsup_{r_1,r_2\to \infty}\alpha\big(\ell_{B_\ell(r_1)},\rho_{B_\ell(r_2)}\big)>0.
$$
By the assumption that $\lim_{r_1,r_2\to \infty}\alpha\big(\rho_{B_\ell(r_1)},\rho_{B_\ell(r_2)}\big)=0$, we may find some $r_0$ so that $\alpha\big(\rho_{B_\ell(r_1)},\rho_{B_\ell(r_2)}\big)<\tfrac \epsilon3$ for all $r_1,r_2> r_0$, whereby, for any $r_1> r_0$, there are arbitrarily large $r_2>r_1$ so that 
$$
\alpha\big(\ell_{B_\ell(r_2)},\rho_{B_\ell(r_1)}\big)>\tfrac \epsilon3.
$$
Choosing $r_0$ large enough, we may assume that $G$ is generated by the open ball $B_\ell(r_0)$.

Let now $1>\zeta_1>\zeta_2>\ldots>\tfrac 12$ be some fixed decreasing sequence. Then, by Lemma \ref{smoothed vs rho}, we may choose $r_0<r_1<r_2<\ldots$ with $\lim_nr_n=\infty$ so that, for all $n\geqslant 1$,
$$
\alpha\big(\ell_{B_\ell(r_n)},\rho_{B_\ell(r_{n+1})}\big)>\tfrac \epsilon3,
$$
while
$$
r_0\rho_{B_\ell(r_0)}\;\succsim\; \ell_{B_\ell(r_0)}
\;\succsim\; \zeta_1r_1\rho_{B_\ell(r_1)}\;\succsim\; \zeta_1\ell_{B_\ell(r_1)}
\;\succsim\; \zeta_2r_2\rho_{B_\ell(r_2)}\;\succsim\; \zeta_2\ell_{B_\ell(r_2)}
\;\succsim\; \ldots.
$$
Observe also that $\ell\leqslant \ell_{B_\ell(r)}$ for all $r$. And since, by \cite[Lemma 2.70]{Rosendal}, $ \ell_{B_\ell(r_0)}$ is a maximal length function on $G$, we have $\ell_{B_\ell(r_0)}\lesssim K\ell$ for some constant $K$. It thus follows that
$$
K\ell\;\succsim\; \ell_{B_\ell(r_0)}
\;\succsim\; \zeta_1r_1\rho_{B_\ell(r_1)}\;\succsim\; \zeta_1\ell_{B_\ell(r_1)}
\;\succsim\; \zeta_2r_2\rho_{B_\ell(r_2)}\;\succsim\; \zeta_2\ell_{B_\ell(r_2)}
\;\succsim\; \ldots \;\succsim\;\tfrac 12\ell.
$$

\begin{claim}
$$
\limsup_{x\to \infty}\frac{\zeta_{n+1}r_{n+1}\rho_{B_\ell(r_{n+1})}(x)}{\zeta_{n-1}r_{n-1}\rho_{B_\ell(r_{n-1})}(x)}\;\;\conv{}{n\to \infty}\;\;1.
$$
\end{claim}

\begin{proof}[Proof of claim]
Suppose for a contradiction that there is some $\lambda<1$ and infinitely many $n$ so that
$$
\limsup_{x\to \infty}\frac{\zeta_{n+1}r_{n+1}\rho_{B_\ell(r_{n+1})}(x)}{\zeta_{n-1}r_{n-1}\rho_{B_\ell(r_{n-1})}(x)}\;\;<\;\;\lambda,
$$
whereby
$$
\lambda{\zeta_{n-1}r_{n-1}\rho_{B_\ell(r_{n-1})}}\;\succsim\; \zeta_{n+1}r_{n+1}\rho_{B_\ell(r_{n+1})}
$$
for these infinitely many $n$. It thus follows that, for any $k\geqslant 1$, we may find some $n<m$ so that
$$
\lambda^kK\ell\;\succsim\;\lambda^k{\zeta_{n}r_{n}\rho_{B_\ell(r_{n})}}\;\succsim\; \zeta_{m}r_{m}\rho_{B_\ell(r_{m})}\;\succsim\;\tfrac 12
\ell.
$$
Choosing $k$ large enough so that $\lambda^kK<\tfrac 12$, we have a contradiction.
\end{proof}

As a consequence of the claim and the assumption of the theorem, we may find some $n\geqslant 1$ large enough so that
$$
\exp(-\tfrac\epsilon8)
\;<\;
\limsup_{x\to \infty}\frac{\zeta_{n+1}r_{n+1}\rho_{B_\ell(r_{n+1})}(x)}{\zeta_{n-1}r_{n-1}\rho_{B_\ell(r_{n-1})}(x)}
\;\leqslant\;
\exp(\tfrac\epsilon8)\cdot\liminf_{x\to \infty}\frac{\zeta_{n+1}r_{n+1}\rho_{B_\ell(r_{n+1})}(x)}{\zeta_{n-1}r_{n-1}\rho_{B_\ell(r_{n-1})}(x)}.
$$
It thus follows that
\mathes{
\exp(-\tfrac\epsilon4)
&< \liminf_{x\to \infty}\frac{\zeta_{n+1}r_{n+1}\rho_{B_\ell(r_{n+1})}(x)}{\zeta_{n-1}r_{n-1}\rho_{B_\ell(r_{n-1})}(x)}\\
&\leqslant \liminf_{x\to \infty}
\frac
{\zeta_{n+1}r_{n+1}\rho_{B_\ell(r_{n+1})}(x)}
{\zeta_{n}\ell_{B_\ell(r_{n})}(x)}\\
&\leqslant \limsup_{x\to \infty}
\frac
{\zeta_{n+1}r_{n+1}\rho_{B_\ell(r_{n+1})}(x)}
{\zeta_{n}\ell_{B_\ell(r_{n})}(x)}\\
&\leqslant 1.
}
We therefore conclude that
\mathes{
\tfrac \epsilon3
&<\alpha\big(\ell_{B_\ell(r_n)},\rho_{B_\ell(r_{n+1})}\big)\\
&=\alpha\big(\zeta_n\ell_{B_\ell(r_n)},\zeta_{n+1}r_{n+1}\rho_{B_\ell(r_{n+1})}\big)\\
&=\log \limsup_{x\to \infty}
\frac
{\zeta_{n+1}r_{n+1}\rho_{B_\ell(r_{n+1})}(x)}
{\zeta_{n}\ell_{B_\ell(r_{n})}(x)}
-\log\liminf_{x\to \infty}
\frac
{\zeta_{n+1}r_{n+1}\rho_{B_\ell(r_{n+1})}(x)}
{\zeta_{n}\ell_{B_\ell(r_{n})}(x)}\\
&\leqslant \log1-\log\exp(-\tfrac\epsilon4)\\
&=\tfrac\epsilon4,
}
which is absurd. This contradiction finishes the proof of the proposition.
\end{proof}

\bibliographystyle{amsalpha} 
\bibliography{ref} 

\end{document}

\section{Asymptotically spherical spaces}
\begin{defn}
A map $X\maps{\phi} M$ between two unbounded metric spaces is said to be an {\em asymptotic dilation} if 
\begin{enumerate}
\item $\phi$ is {\em controlled}, that is, for all $r>0$, there is $R>0$ so that
$d(\phi x,\phi y)\leqslant R$ whenever $d(x,y)\leqslant r$, and 
\item the limit 
$$
\lambda=\lim_{d(x,y)\to \infty}\frac {d(\phi x,\phi y)}{d(x,y)}
$$
exists and $\lambda>0$. 
\end{enumerate}
\end{defn}